\documentclass[11pt]{amsart}

\usepackage[utf8]{inputenc}
\usepackage{amsmath,amssymb,mathrsfs, amssymb}
\usepackage{enumerate}
\usepackage[all,cmtip]{xy}
\usepackage{color}
\usepackage{mathtools}

\newcommand{\Z}{\mathbb Z}
\newcommand{\Q}{\mathbb Q}

\newcommand{\C}{\mathbb C}
\newcommand{\N}{\mathbb N}

\newcommand{\bE}{\mathbb E}

\newcommand{\bM}{\mathbb M}
\newcommand{\PP}{\mathbb{P}}

\newcommand{\dt}{\bullet}

\newcommand{\cE}{\mathcal E}

\newcommand{\cO}{\mathcal O}

\newcommand{\cF}{\mathcal F}
\newcommand{\cL}{\mathcal L}

\newcommand{\fM}{\mathfrak M}

\newcommand{\fP}{\mathfrak P}

\newcommand{\uc}{\tilde{c}}

\newcommand{\cbar}{\overline{c}}

\newcommand{\GL}{\mathrm{GL}}
\newcommand{\PGL}{\mathrm{PGL}}

\newcommand{\Gm}{\mathbb{G}_m}

\renewcommand{\ss}{\mathrm{ss}}
\renewcommand{\SS}{\mathrm{SS}}
\newcommand{\St}{\mathrm{S}}

\newcommand\Hom{\operatorname{Hom}}

\newcommand\Spec{\operatorname{Spec}}

\newcommand\Quot{\operatorname{Quot}}

\newcommand{\cosk}{\operatorname{cosk}}

\numberwithin{equation}{section}

\newtheorem{theorem}{Theorem}[section]
\newtheorem{lemma}[theorem]{Lemma}
\newtheorem{proposition}[theorem]{Proposition}
\newtheorem{cor}[theorem]{Corollary}

\theoremstyle{remark}

\theoremstyle{definition}

\begin{document}

\title{Universal Chern classes on the moduli of  bundles}

\author{Donu Arapura}
\thanks{Author supported by a grant from the Simons foundation}
\address{Department of Mathematics\\ Purdue University\\
West Lafayette, IN 47907-2067}
\email{arapura@purdue.edu}

\date{\today}



\maketitle

\begin{abstract}
 The goal of this paper is to construct  universal cohomology classes  on the moduli space of stable bundles  over a curve
 when it is not a fine moduli space, i.e. when the rank and degree are not coprime. More precisely, we show that certain Chern classes of the universal bundle
 on the product of the curve with the moduli stack of bundles lift to the product of the curve with the moduli space of stable
 bundles.
 \end{abstract}

\section{Introduction}

This is an offshoot of some joint work with Richard Hain \cite{ah}. In both papers, we are concerned
with the moduli space $\fM_n(C,L)$ of semistable bundles of rank $n>1$ and fixed determinant $L$ on a smooth
projective curve $C$ of genus $g>1$.  The focus here is on the bad case, when $n$ and $\deg L$ are not coprime.
Then  $\fM_n(C,L)$ is usually
singular, and there is no universal or  Poincar\'e vector bundle on 
$C\times \fM_n(C,L)$ \cite{ramanan}. Nevertheless, it makes sense to  ask whether there are universal
Chern classes. One way to formulate this precisely is to consider the moduli stack
$\bM^{ss}_n(C,L)$ of  rank $n$ semistable bundles on  $C$ with  determinant $L$.  By its very nature, there is a universal
bundle $\bE^{ss}$ on  $C\times \bM^{ss}_n(C,L)$. There is also a $1$-morphism $ \bM^{ss}_n(C,L)\to \fM_n(C,L)$. The  question
is whether the Chern classes $c_j(\bE^{ss})$ lift to $H^*(C\times \fM_n(C,L))$. In \cite{ah}, we show that $c_2(\bE^{ss})$ does not
lift when $L$ is trivial, so the  answer is negative in general. Nevertheless, in this note we will show that the answer is positive over the moduli of
{\em stable} bundles $\fM^s_n(C,L)\subset \fM_n(C,L)$ for certain combinations of Chern classes such as
$$\left.\left(c_2(\bE^{ss})- \left(\frac{n-1}{2n}\right) c_1(\bE^{ss})^2\right)\right|_{\fM^s}$$
that we call reduced Chern classes.
The proof is not difficult, and it uses the fact that although the Poincar\'e bundle need not exist, its associated projective space
bundle does.  

The main theorem is formulated and proved without any mention of stacks.
The stack  viewpoint  in only discussed in the epilogue.
My thanks to Dick Hain for many discussions about this material. Also thanks are due  to Kapil Paranjape for an insightful comment
during my talk on this stuff at TIFR in March 2023. Finally, my thanks to the  referee for several helpful remarks.

\section{Some background on moduli spaces}

Fix a curve $C$ over an algebraically closed field $k$, an integer $n>1$ and a line bundle $L$ as in the introduction. Let $d=\deg L$.
Denote by $\SS$ the contravariant functor from $k$-schemes to sets that takes a scheme $Y$ to the set
$$
\SS(Y) :=
\left\{
\parbox{2.75in}{vector bundles $E$ of rank $n$ over $C\times Y$ whose restriction to $C\times\{y\}$ is semistable
for each $y\in Y$, with   $\det (E)$  Zariski locally isomorphic to $pr_1^*L$ }
\right\}/\cong
$$

This functor is coarsely represented by a projective variety $\fM_n(C,L)$. For the purposes of this paper,
it means that there is a natural transformation $\eta : \SS(-) \to \Hom(-,\fM_n(C,L))$ 
which is initial among all natural transformations  $\SS(-) \to \Hom(-,Z)$. 
Since $\Hom(-,\fM_n(C,L))$ is a sheaf on the big Zariski site,
we find that $\eta$ factors through the sheafification of $\SS(-)$. It follows that $\SS(Y) \to \Hom(Y,\fM_n(C,L))$
factors through the quotient $\SS(Y)/\!\sim$, where $E'\sim E$ if $E'\cong E\otimes pr_2^*M$, for some $M\in Pic(Y)$.
The resulting map $\SS(-)/\!\sim\, \to \Hom(-,\fM_n(C,L))$ is still not  an isomorphism of functors in general.
In fact 
$$\SS(\Spec k)= \SS(\Spec k)/\!\sim \, \to Hom(\Spec k,\fM_n(C,L))$$
is not bijection when $n$ and $d$ are not coprime. Given $E\in \SS(Y)$, let $\eta_E:Y\to \fM_n(C,L)$ denote the corresponding morphism.
 The functor $\St(-)$ is defined as above by changing  the word ``semistable" to ``stable". This functor is coarsely represented by
 an open subvariety $\fM_n^s(C,L)\subset \fM_n(C,L)$.

We  need to recall some details for the construction of the above moduli spaces.
Further details  can be found in  \cite{pic,seshadri}. Fix an ample line bundle $\cO_C(1)$ on $C$. We can choose  $m\gg 0$ so for each semistable rank $n$ vector bundle $E$ over $C$ with
$\det E\cong L$, $E(m)$ is globally generated and has vanishing first cohomology. Set
$$
h(t) = n(1-g) +d+ [n\deg \cO_C(1)] t
$$
$$  N= h(m)$$
Denote by $Q$ the Quot scheme $\Quot(\cO_C(-m)^N, h(t))$ that parameterizes quotient maps  $r:\cO_C(-m)^N\twoheadrightarrow E$ where $E$ is a rank $n$ coherent sheaf
 with Hilbert polynomial $h(t)$, $\det E\cong L$ and such that $r$ induces an isomorphism $H^0(\cO^N)\cong H^0(E(m))$.
  Then $Q$ parameterizes the kernels $\ker r$ of the framings and $\GL_N(k)$ acts on $Q$ by precomposition.
The centre acts trivially, so the action factors through $\PGL_N(k)$.
 Let
$$
Q \supset Q^\ss \supset Q^s
$$
be the open subvarieities parameterizing  semistable and stable vector bundles respectively. These are exactly the sets of semistable and stable points with respect to a 
$GL_N(k)$ linearized ample line bundle $\cL$ on $Q$.
The moduli space and its open subset parameterizing stable vector bundles are the GIT quotients:
$$
\fM_n(C,L) = Q^\ss//\PGL_N(k)$$
$$ \fM_n^s(C,L) =Q^s//\PGL_N(k)
$$
Let $p^{ss}:Q^\ss\to  \fM_n(C,L)$ and $p:Q^s\to  \fM_n^s$(C,L) denote the projections.
The projection $p:Q^s \to \fM_n^s(C,L)$ is a principal $\PGL_N$ bundle.

Let $pr_1:C\times Q\to C$ denote the projection.
Then  $C\times Q$ carries a universal quotient of $pr_1^*\cO_C(-m)^N$, which restricts
to  a $\GL_N(k)$-equivariant vector bundle $\cE$ on $C\times Q^{s}$. This bundle need not descend to  $C\times \fM^s_n(C,L)$.
However, the  associated projective space bundle will. Recall that
a Brauer--Severi scheme over $X$ is a projective space bundle which is locally trivial in the \'etale topology.
If $E$ is a vector bundle on $X$, then $\PP(E)$ is a Brauer--Severi scheme. However, the converse is not true
unless the Brauer group of $X$ is trivial \cite{grothendieck}.

\begin{proposition}[\cite{bbgn}]
There exists a  Brauer--Severi scheme  $\fP$ over  $C\times \fM_n^s(C,L)$,
such that $\PP(\cE)=p^{*}\fP$.
\end{proposition}

\begin{proof} Since the proof is omitted in \cite{bbgn}, we indicate it here.
Let $\fP'= \PP(\cE)$ and $G= PGL_N(k)$.
The action of $GL_N(k)$ on $\cE$  induces a free action of $G$ on $\fP'$. 
This action, together with the isomorphism $G\times Q^s\cong Q^s\times_{\fM_n^s} Q^s$,
induces descent data on $\fP'$. By faithfully flat descent \cite[pp 135-136, theorem 6]{blr}, there exists a scheme
$\fP$ over  $C\times \fM_n^s(C,L)$
 whose pull back is $\fP'$ (more concretely $\fP= \fP'/\PGL_N(k)$).
It follows that the (closed) fibres of $\fP\to C\times \fM_n^s$ are isomorphic to $\PP^{n-1}$. Therefore it is a Brauer--Severi scheme by 
\cite[Cor 8.3]{grothendieck}.

\end{proof}

\section{Reduced Chern classes}\label{section:reducedChern}

Let us modify the question given in the introduction as follows:
{\em For which polynomials $A(c_1, \ldots, c_n)\in \Q[c_1,\ldots c_n]$ 
can we find universal classes $\alpha\in H^*(C\times \fM_n(C,L),\Q)$
such that for $E\in \SS(Y)$, 
$$A(c_1(E), \ldots, c_n(E)) = \eta_E^*\alpha?$$}
We do not know the answer. However, we will answer the analogous  question on $\fM_n^s(C)$.

Let us first treat the case where $k=\C$. Schemes over $\C$ of finite type  will be given the analytic topology.
One thing to observe is that for any line bundle $M$ on $Y$,
$\eta_{E}= \eta_{E\otimes M}$. This puts strong restrictions of the sorts of
polynomials we can allow, and  it suggests the following definition.
 Given a  rank $n$ complex vector bundle $E$ on a topological space $X$,
define the $r$th reduced Chern class of level $n$ informally by
$$\cbar_r(E) = c_r(E\otimes \det(E)^{-1/n})\in H^{2r}(X,\Q)$$
More precisely,
$$\cbar_r(E) := \sigma_r(e_1- \frac{e_1+\dots+ e_n}{n},\ldots, e_n- \frac{e_1+\dots+ e_n}{n})$$
where  $e_1,\dots, e_n$ are the Chern roots, and  $\sigma_r$ is the $r$th elementary symmetric polynomial.
Since $n$ will be fixed, we suppress it from the notation.
By \cite[ex 3.2.2]{fulton}, we obtain a formula
\begin{equation}\label{eq:cbar}
 \cbar_r(E)= \sum_{i=0}^r (-1)^{r-i}\frac{1}{n^{r-i}}\binom{n-i}{r-i} c_1(E)^{r-i}c_{i}(E)
\end{equation}
In particular, the first  two values are
$$\cbar_1(E)=0$$
$$\cbar_2(E)= c_2(E)- \left(\frac{n-1}{2n}\right) c_1(E)^2$$ 

These classes have a simple characterization.

\begin{lemma}
The $j$th  reduced Chern class satisfies
\begin{equation}\label{eq:cbar1}
 \cbar_j(E) = c_j(E)  \quad \text{ when } c_1(E)=0
\end{equation}
and
\begin{equation}\label{eq:cbar2}
 \cbar_j(E\otimes L)= \cbar_j(E)
\end{equation}
for any rank $n$ vector bundle $E$ and any line bundle $L$.
Conversely, $\cbar_j$ is the only polynomial in $c_1,\ldots, c_n$ satisfying these conditions for all CW complexes $X$.
\end{lemma}

\begin{proof}
 We just prove the last statement, since the first  follows immediately from the definition of $\cbar_j$.
 Suppose that $q(c_1,\ldots, c_n) $ is a polynomial satisfying   \eqref{eq:cbar1} and  \eqref{eq:cbar2}  for  all $X$.
 It suffices to treat the universal case where $X=BGL_n(\C)$ is the classifying space and $E$ is the universal bundle over it.
 Then $q$ can be identified with the cohomology class it evaluates to.
 We write $q$ as an symmetric polynomial in the roots $e_1,\ldots, e_n$.
 Condition \eqref{eq:cbar2}, with $L= (\det E)^t$,  implies
 $$q(e_1 + t(e_1+\ldots +e_n),\ldots, e_n + t(e_1+\ldots +e_n)) -q(e_1,\ldots, q_n)=0$$
 for every $t\in \Z$.  Treating this as a polynomial in $t$ with coefficients in the field of symmetric rational
 functions in $e_1,\ldots, e_n$, it follows that it must be identically zero.
 Therefore
  $$q(e_1,\ldots, e_n)= q(e_1 -\frac{1}{n}(e_1+\ldots +e_n),\ldots, e_n -\frac{1}{n}(e_1+\ldots +e_n)) $$
 Condition \eqref{eq:cbar1} means that 
  $$q = c_j+ rc_1 =\sigma_j + r\sigma_1$$
  for some symmetric  polynomial  $r$. Therefore
  $$q(e_1,\ldots, e_n)=q(e_1 -\frac{1}{n}(e_1+\ldots +e_n),\ldots)= \sigma_j(e_1 -\frac{1}{n}(e_1+\ldots +e_n),\ldots)$$
  as claimed, because $\sigma_1$ will vanish for the above argument.
\end{proof}
The lemma implies that $\cbar_j(E)$ is an invariant of the projective space bundle $\PP(E)$. This can be viewed as a $PGL_n(\C)$-bundle.
We will extend these classes to an arbitrary  $PGL_n(\C)$-bundle  (which need not arise from a vector bundle \cite{grothendieck}).

\begin{lemma}
 The cohomology ring 
 $$H^*(B\PGL_n(\C), \Q) = \Q[\cbar_2,\ldots, \cbar_n] $$
 where $\cbar_j$ is a class of degree $2j$.
 Under the natural map $B\GL_n(\C)\to B\PGL_n(\C)$, the classes $\cbar_i$ pull back to the reduced Chern classes  $\cbar_i(E) $ of the universal bundle.
\end{lemma}

\begin{proof}
 First note that
under the map $H^*(BGL_n(\C),\Q)\to H^*(BSL_n(\C),\Q)$, we can identify the cohomology ring
$$H^*(BSL_n(\C), \Q) =\Q[c_1,\ldots, c_n]/(c_1) \cong  \Q[\cbar_2,\ldots, \cbar_n]$$
The exact sequence
$$1\to \Z/n\Z\to SL_n(\C)\to PGL_n(\C)\to 1$$
induces a fibration
$$ BSL_n(\C)\to BPGL_n(\C)$$
with fibre $B\Z/n\Z$.
Therefore, we have an isomorphism
$$H^*(BPGL_n(\C), \Q) \cong H^*(BSL_n(\C),\Q)$$
  The homomorphism $ GL_n(\C)\to PGL_n(\C)$ induces
an algebra homomorphism
 $$\Q[\cbar_2,\ldots, \cbar_n]\to \Q[c_1,\ldots, c_n]$$
 which is a section for the projection given above.
The classes $\cbar_i$ map to the reduced Chern classes  $\cbar_i(E) $ of the universal bundle by the previous lemma.

\end{proof}

Given a $PGL_n(\C)$-bundle, or equivalently a $\PP^n$-bundle $P$ on $X$,
we obtain a classifying map $f:X\to BPGL_n(\C)$.
We define 
$$\cbar_i(P) = f^*\cbar_i$$
We have that
$$\cbar_i(\PP(E))= \cbar_i(E)$$
by the previous  lemma.

\section{Reduced Chern classes for Brauer-Severi schemes}\label{section:BS}

Now suppose that $k$ is an arbitrary algebraically closed field.
Choose  $\ell$ coprime to $\operatorname{char} k$.
Given a vector  bundle $E$ on a $k$-variety $X$, 
the Chern classes $c_i(E)\in H_{et}^{2i}(X, \Q_\ell)$ with values in $\ell$-adic cohomology  can be defined by Grothendieck's procedure
 (see \cite{groth-chern, milne} or the next section).
We turn now to the definition of reduced Chern classes for Brauer-Severi schemes.
For scheme of the form $\PP(E)$, can define $\cbar_i(E)$ using \eqref{eq:cbar}, and check that this depends
only on $\PP(E)$. This method does not generalize to Brauer-Severi schemes.
There are a couple of ways to proceed. One is replace the classifying space $BPGL_n$ by a simplicial scheme or stack.
This  approach is briefly discussed in the last section.
In this section, we will follow a more down to earth path, and redo the construction for $\PP(E)$ in a way that does generalize.
Rather than working with $\cO_{\PP(E)}(1)$, which does not exist on a general Brauer-Severi scheme, we work with the  relative dualizing sheaf  
which does. A standard calculation shows that
$$\pi_*\omega_{\PP(E)/X}^{-1} \cong S^n(E)\otimes \det(E)^{-1}$$

We start with a purely algebraic lemma.

\begin{lemma}
 In the polynomial ring $R=\Q[x_1,\ldots, x_n]$, let 
 $$s_i = \sigma_i(y_1,\ldots, y_N),$$
 where 
 $$\{y_1,\ldots, y_N\} = \{m_1 x_1+\ldots +m_n x_n\mid m_i\in \N, m_1+\ldots + m_n= n\}$$
 Then the set  $\{s_1,\ldots, s_n\}$ generates the ring  of invariant polynomials $R^{S_n}$.
 
 \end{lemma}

\begin{proof}
We start by  some notation. We assume that $y_i$ bijectively parameterize the expressions $m_1 x_1+\ldots +m_n x_n$,
such that $y_i=nx_i$ for $i\le n$. In particular,  $N= \binom{2n-1}{n}$.
A partition of length $n$ is an $n$-tuple of  nonincreasing integers in $\N^n$.
We reserve the letters $\lambda= (\lambda_1,\lambda_2,\ldots)$ and $\mu=(\mu_1,\mu_2,\ldots)$ for partitions.
When indicating a partition, we usually just list the nonzero integers in it.
We write $(1^r) = (1,\ldots, 1)$ with $r$ $1$'s.
We order the set of partitions using reverse lexicographic order, so for example
$$(2, 1) > (2) > (1^3)$$
Let $P_d$ denote the  set of partitions $\lambda$ of length  $n$ and  weight $|\lambda| = \lambda_1+\ldots + \lambda_n= d$.


The vector space of degree $d$ homogenous invariants $R_d^{S_n}$ has two bases that we need to recall \cite{macdonald}. The first  basis is familiar from basic algebra. It is the
 set of monomials
$$e_\lambda = e_{\lambda_1}e_{\lambda_2}\ldots$$
where
$$e_r = \sigma_r(x_1,\ldots, x_n)$$
and
$\lambda\in P_d$.
The second basis is  the set of monomial symmetric functions
$$m_\lambda = \sum_\sigma x_1^{\sigma(\lambda_1)}\ldots x_n^{\sigma(\lambda_n)} $$
where the sum runs over distinct permutations of $(\lambda_1,\ldots, \lambda_n)$. The index $\lambda$ again ranges over $P_d$.
One has that
$$e_r = e_{(r)}= m_{(1^r)}$$
In general, there is  a change of basis formula \cite[chap I (1.10), (2.3)]{macdonald}
$$ e_{\lambda} = m_{\lambda'} + \sum_{\mu<\lambda'} a_{\lambda'\mu}m_\mu$$
where $\lambda'$ is the conjugate partition. 
 
Let us say that a polynomial in $R$ is positive if it has no negative coefficients.
We claim that a positive polynomial in $R^{S_n}$ is a linear combination of $m_\lambda$'s with
nonnegative coefficients. Clearly it is enough to prove this for a nonzero invariant polynomial $p$
with integer coefficients. If $x_1^{\alpha_1}\ldots x_n^{\alpha_n}$ appears in $p$ with 
coefficient $c>0$, then all monomials with exponent a permutation of $\alpha=(\alpha_1,\ldots, \alpha_n)$
occur in $p$ with the same coefficient. In particular, we can assume that $\alpha$ is a partition,
and that
$$p = cm_\alpha + q$$
where $q$ is a positive invariant polynomial with integer coefficients. Therefore the claim follows by induction
on the sum of coefficients.

Let us say that a  polynomial $\sum_i p_i t^i\in R[t]$ is invariant or positive if all the coefficients $p_i$ have this property.
 We can write
 $$\sum_{i=0}^N s_i t^i = \prod_{i=1}^N( 1+ y_it) =  \prod_{j=1}^n (1+ nx_jt)\underbrace{\prod_{i=n+1}^N ( 1+ y_it)}_F $$ 
 Clearly $F$ is a positive invariant polynomial with a constant term of $1$.
 It  follows  from this and the previous formulas that 
 $$s_r = c_r e_r + \sum_{\lambda\in P_r-\{(r)\}} d_{r,\lambda } e_\lambda, \quad r\le n$$
where $c_r>0$. The system of equations is ``triangular" in the sense that
$$s_1 = c_1 e_1$$
$$s_2 = c_2 e_2 + d_{2,(1,1)} e_1^2$$
$$ s_3 = c_3 e_3 + d_{3, (2,1)} e_2e_1 + d_{3, (1,1,1)} e_1^3 $$
$$\ldots$$
So we can solve
$$e_1 = \frac{1}{c_1} s_1$$
$$e_2 = \frac{1}{c_2} s_2 - \frac{d_{2,(1,1)}}{c_1^2} s_1^2$$
$$\ldots$$
$$e_i = \psi_i(s_1,\ldots, s_n)$$
for some rational polynomial $\psi_i$.
and the lemma is proved.
\end{proof}


\begin{lemma}\label{lemma:SnE}
 The Chern classes of $F=S^n(E)\otimes \det(E)^{-1}$ are given by  rational polynomials in  $\cbar_i(E)$.
 Conversely, there exists universal rational polynomials $\phi_i$ such that
 $$\cbar_i(E) = \phi_i(c_2(F),\ldots, c_n(F))$$
\end{lemma}

\begin{proof}
Let $e_1, \ldots, e_n$ denote the Chern roots of $E$. Then
 $$f_i = e_i - \frac{e_1+\dots+ e_n}{n}$$
 are the roots of the symbolic vector bundle $E\otimes \det(E)^{-1/n}$. The set of Chern roots of 
 $F=S^n(E)\otimes \det(E)^{-1}$ is  $\{m_1f_1+\ldots +m_n f_n\mid \sum m_i=n\}$.
 This implies the first statement.
 Furthermore, in the notation of the previous proof,
 $$ \cbar_i(E) = \psi_i(0, c_2(F),\ldots, c_n(F))$$

\end{proof}

Let $\pi:P\to X$ be a Brauer-Severi  scheme with $\PP^{n-1}$ as its geometric fibres.
Then $F= \pi_*\omega_{P/X}^{-1}$ is a locally free sheaf with respect to the \'etale topology.
Therefore it is locally free  with respect to the Zariski topology (by faithfully flat descent \cite[p 134]{blr}).
Now we define the reduced Chern classes of $P$ by
\begin{equation}\label{eq:brauerC_i}
 \cbar_i(P) = \phi_i(c_2(\pi_*\omega_{P/X}^{-1}),\ldots, c_n(\pi_*\omega_{P/X}^{-1}))
\end{equation}
The previous lemma implies that
$$\cbar_i(\PP(E))= \cbar_i(E)$$

\section{Chern classes of bundles on simplicial schemes}

There is one more ingredient that goes into  the proof main theorem that we need to discuss.
 We start by summarizing a few facts about simplicial  schemes \cite{deligne-hodge3, friedlander}.
  A simplicial scheme is given by a sequence of
 schemes $X_0, X_1,\ldots$ and face maps $\delta_i:X_n\to X_{n-1}$ and degeneracy maps $s_i:X_{n-1}\to X_n$
 satisfying some standard identities.
  If $X_\dt$ is simplicial scheme, an \'etale  sheaf $\cF_\dt$  on it is  given by  a collection of sheaves $\cF_n$ on $X_n$
together with structure maps $\alpha^* \cF_n\to \cF_m$ for the various face and degeneracy maps $\alpha:X_m\to X_n$.
These structure maps are subject to the appropriate compatibility conditions.
Similarly one can define a sheaf on the corresponding simplicial analytic space $X_\dt^{an}$  when $X_\dt$ is a 
simplicial object in the category of $\C$-schemes of finite type.
For either topology, one has the  constant sheaf $\Z_{X_\dt}$.
The main example of a  simplicial scheme that we will  need is the coskeleton of a morphism $U\to X$ given by
$$
\cosk( U\to X)_\dt=\xymatrix{
\ldots & U\times_X U\ar@<-1ex>[r]\ar[r]\ar@<1ex>[r] & U\ar[l] },
$$
where the maps are projections and diagonal maps.
This comes with an augmentation $\cosk(U\to X)_\dt\to X$. A sheaf on  $\cosk(U\to X)_\dt\to X$ can be obtained by pulling back a sheaf from $X$.
If $\cF_{\dt}$ is a sheaf on  a simplicial scheme $X_\dt$  in  either topology, cohomology is defined by
$$H^i(X_\dt, \cF_\dt) = Ext^i(\Z_{X_\dt}, \cF_\dt)$$ 

\begin{lemma}\label{lemma:descent}
 If $U\to X$ is the disjoint union of the sets of a Zariski open cover $\mathcal{U}$, then the cohomology, in either topology, of $\cosk(U\to X)_\dt$
 with coefficients in the pull back of a sheaf $\cF$ from $X$  is isomorphic to $H^i(X,\cF)$.
\end{lemma}

\begin{proof}
 See \cite[prop 3.7]{friedlander} for the \'etale case, and  \cite[(5.3.7)]{deligne-hodge3} for the analytic. 
\end{proof}

A vector bundle (respectively Brauer-Severi scheme) over $X_\dt$ is a morphism of simplicial schemes $E_\dt\to X_\dt$ such that $E_n\to X_n$ is a vector bundle (respectively Brauer-Severi scheme) in the usual sense, and  such that the face and degeneracy maps fit into Cartesian diagrams
$$
\xymatrix{
 E_n\ar[r]\ar[d] & E_m\ar[d] \\ 
 X_n\ar[r] & X_m
}
$$
 Grothendieck's procedure for defining Chern classes   \cite{groth-chern, milne},
can be easily generalized to simplicial schemes. Here we briefly outline the construction.
Suppose that $\ell$ is prime and $X_\dt$ is a simplicial scheme over $\Z[1/\ell]$, then we have  Kummer sequence 
$$1\to \mu_{\ell^n}\to \mathbb{G}_m\to \mathbb{G}_m\to 1$$
on it \cite[p 66]{milne}.
The connecting map defines a first Chern class map
$$c_1: Pic(X_\dt)= H^1(X_\dt,\mathbb{G}_m) \to \varprojlim_n H_{et}^2(X_\dt, \mu_{\ell^n})=  H_{et}^2(X_\dt, \Z_\ell(1))$$
Let us suppose from now on that  $X_\dt$ is a simplicial variety over an algebraically closed field $k$ of characteristic different
from $\ell$. Then we can ignore the twist by identifying  $\Z_\ell(1)\cong \Z_\ell$.
Now suppose that $E_\dt$ is a simplical vector bundle of rank $n$ on  $X_\dt$, then $\PP(E_\dt)\to X_\dt$
is a  Brauer-Severi scheme. The collection of  tautological  bundles $\cO_{\PP(E_\dt)}(1)$ defines line bundle on $L_\dt$
on $\PP(E_\dt)$.

\begin{lemma}
There is an isomorphism of graded $H_{et}^*(X,\Q_\ell)$-modules
 $$H_{et}^*(\PP(E_\dt), \Q_\ell) \cong \bigoplus_{i=0}^{n-1}H_{et}^{*-2i}(X_\dt,\Q_\ell)\cup c_1(L_\dt)^i$$
\end{lemma}

\begin{proof}
 By \cite[prop 2.4]{friedlander} there is a morphism of spectral sequences
 $$
\left.
\begin{array}{ccc}
 E_1 = \bigoplus_{i=0}^{n-1} H^{q-2i}(X_p,\Q_\ell)\cup c_1(L_p)^i & \Rightarrow &  \bigoplus_{i=0}^{n-1} H^{p+q-2i}(X_\dt, \Q_\ell)\cup c_1(L_\dt)^i \\ 
 \downarrow &  & \downarrow \\ 
  E_1 = H^q(\PP(E_p),\Q_\ell) & \Rightarrow & H^{p+q}(\PP(E_\dt), \Q_\ell)
\end{array}
\right.
$$ 
This is an isomorphism on the $E_1$ pages of the spectral sequences by a standard argument \cite[pp 272-273]{milne}. Therefore it is an isomorphism on the abutments.

 \end{proof}
 
Now one can define   higher Chern classes  $c_i(E_\dt)\in H^{2i}(X_\dt, \Q_\ell)$ using the relations
$$c_1(E_\dt) = c_1(\wedge^n E_\dt)$$
$$c_1(L_\dt)^n + \pi^* c_2(E_\dt)c_1(L_\dt)^{n-1} + \ldots =0$$
The same procedure works in the analytic case to define  $c_i(E_\dt)\in H^{2i}(X_\dt^{an},\Q)$,
with the Kummer sequence replaced by the exponential sequence.
We can define the reduced Chern classes $\cbar_i(P_\dt)$ of a Brauer-Severi scheme by using \eqref{eq:brauerC_i}. These classes are functorial,
so as a corollary to lemma \ref{lemma:descent}, we obtain

\begin{cor}
 Let $U\to X$ be the disjoint union of the sets of a Zariski open cover $\mathcal{U}$. Suppose that $E$ is a vector bundle
 on $X$ and let $E_\dt$ denote its pull back to $\cosk(U\to X)_\dt$.
 Under the isomorphism
 $$H^{2i}(\cosk(U\to X)_\dt) \cong H^{2i}(X)$$
 $c_i(E_\dt)$ (respectively $\cbar_i(E_\dt)$) corresponds to  $c_i(E)$ (respectively $\cbar_i(E)$). 
\end{cor}

\section{Main theorem}

\begin{theorem}\label{thm:main1}
 \-
\begin{enumerate}
 \item When $k=\C$,
 there exists Hodge classes $\uc_j\in H^{2j}((C\times \fM_n^s(C))^{an},\Q), j=2,\ldots n$ which are universal reduced Chern classes; more precisely
  given a $\C$-scheme $Y$ of finite type  and $E\in \St(Y)$, $\eta_E^*\uc_j = \cbar_j(E)$.
  \item When $k$ is an arbitrary algebraically closed field, there exists classes $\uc_j\in H_{et}^{2j}(C\times \fM_n^s(C),\Q_\ell), j=2,\ldots n$  such that 
 for a $k$-scheme $Y$ and $E\in \St(Y)$, $\eta_E^*\uc_j = \cbar_j(E)$.
\end{enumerate}
   
\end{theorem}

\begin{proof}
We define $\uc_j =\cbar_j(\fP)$. When $k=\C$, this is a Hodge class  of weight $2j$ by \cite[Cor.~9.1.3]{deligne-hodge3}.
 The object $E$  is a family of semistable bundles over $C\times Y$. 
Suppose that we can lift $\eta_E$ to a morphism $g:Y\to Q^s$ such that $g^*\cE \cong E$.  Then we immediately get that
 $$\cbar_j(E) = g^*\cbar_j(\cE) = \eta_E^*\uc_j$$
 as desired.
 
In general, the lifts only exist locally so the argument is a bit more involved. To be more precise,
 we can choose a Zariski open cover $\{U_i\}$ of $Y$ and  surjections $p_1^*\cO_C(-m)|_{U_i}^N\to E|_{C\times U_i}$,
 where $m,N$ are constants used in the definition of $Q=\Quot(\cO_C(-m)^N, h(t))$. Let $U= \coprod_i U_i\xrightarrow{\pi} Y$.
 By the universal property of $Q^s\subset Q$,
 we obtain a morphism $g:U\to Q^s$ such that $g^*\cE \cong \pi^*E$ and $\eta_E\circ \pi = p\circ g$.
 This extends to a map of simplicial schemes
$$g_\dt: \cosk(U\to Y)_\dt \to \cosk(Q^s\to\fM_n^s(C))_\dt $$
 such that $g_\dt ^*\fP_\dt  \cong \PP(E_\dt)$, where $\fP_\dt $ (respectively $E_\dt$)
  denotes the Brauer-Severi scheme (respectively vector bundle) obtained by pulling 
  back $\fP$ (respectively $E$) to the simplicial scheme on the right (left).
 Therefore
 $$\cbar_j(E) = \cbar_j(E_\dt) =  g^*\cbar_j(\fP_\dt) = \eta_E^*\uc_j$$

\end{proof}

\section{Interpretation using stacks}

We want to explain the interpretation of the main theorem in terms of stacks described in the introduction.
We will not define stacks here, but instead refer to  \cite{heinloth} for the relevant background.
But in very rough terms, a major difference between schemes and stacks, is that the former  represent 
set valued functors and the latter  groupoid valued
functors.  This avoids some problems of nonrepresentability of moduli functors by schemes caused by automorphisms.
 In particular, we have stacks $\bM_n^s(C,L)$ (respectively $\bM_n^\ss(C,L)$)  which represents the functor which to any $k$-scheme $T$
assigns the groupoid of families of stable (resp. semistable) vector bundles of rank $n$ with determinant $L$ on $C$ parameterized by $T$.
These can be constructed  as quotient stacks
$\bM_n^s(C,L)=[Q^s/GL_N(k)]$ and $\bM_n^\ss(C,L)=[Q^\ss/GL_N(k)]$.
(To avoid any confusion, recall that $N$ is  not $n$ but rather the dimension of the space of global sections of a suitable twist of $E\in Q^\ss$.)
Also by construction,
we get  induced $1$-morphisms $ \bM^{s}_n(C,L)\to \fM^s_n(C,L)$ and $ \bM^{\ss}_n(C,L)\to \fM_n(C,L)$.
For any $k$-scheme $T$ with $E\in \SS(T)$, $\eta_E$ factors through  $\bM^{\ss}_n(C,L)$.

The $\ell$-adic cohomology of a quotient stack $[X/G]$, such as $\bM_n^s(C,L)$, can be defined as the cohomology of the simplicial
scheme
$$  (EG_\bullet\times X)/G$$
where 
$$EG_\bullet =\ldots G\times G\rightrightarrows G\to \Spec k$$
is constructed in \cite[\S 6.1]{deligne-hodge3}. This is just a version of the Borel construction or homotopy quotient, and so cohomology
of $[X/G]$ is nothing but $G$-equivariant cohomology of $X$. The quotient stack is $[\Spec k/GL_n(k)]$ is called the classifying
stack $\mathbb{B} GL_n(k)$. 
Its cohomology is the polynomial ring $\Q_\ell[c_1,\ldots, c_n]$ in universal Chern classes.
Arguments similar to those used in section \ref{section:reducedChern} show that the cohomology of $\mathbb{B} PGL_n(k)$ is
a polynomial ring in the  universal reduced Chern classes $\cbar_2,\ldots, \cbar_n$.

Let $E=\operatorname{\bf Spec} \cE^\vee \xrightarrow{\pi} C\times Q^s$ denote geometric vector bundle associated to $\cE$.
This has  an action of  $GL_N(k)$ compatible with the action on $C\times Q^s$.
The quotient $\bE=[E/GL_N(k)]$ can be viewed as a rank $n$ vector bundle on $C\times \bM^{s}_n(C,L)$.
The bundle  $\bE$  defines a $1$-morphism $C\times \bM^s(C,L)\to \mathbb{B}GL_n(k)$, the pullbacks of $c_i$ yields the Chern classes
$$c_j(\bE)\in H_{et}^{2j}(C\times \bM_n^s(C,L),\Q_\ell)$$
In particular, we can define the reduced Chern classes $\cbar_j(\bE)$ using \eqref{eq:cbar}.
Theorem \ref{thm:main1}  can be reformulated as:

\begin{theorem}
 The reduced Chern classes $\cbar_j(\bE)\in H_{et}^{2j}(C\times \bM_n^s(C,L),\Q_\ell)$ lift to classes 
$\uc_j\in H_{et}^{2j}(C\times \fM_n^s(C),\Q_\ell)$.
\end{theorem}

Finally to round things out, we describe the precise  relationship between the cohomology of the moduli space and moduli stack.
We can identify $\fM_n^s(C,L) = [Q^s/PGL_n(k)]$. The map $ \bM^{s}_n(C,L)\to \fM^s_n(C,L)$ is a $\mathbb{G}_m$-gerbe \cite[ex 3.7, ex 3.9]{heinloth}, which
for our purposes means that it behaves like a fibration with $B\mathbb{G}_m$ as fibre.
If we disregard the  the multiplicative structure, then it behaves like a product (see \cite[thm 5.4]{dhillon}).

\begin{proposition}
 As a graded vector space, we have an isomorphism
 $$H_{et}^*(\bM^s_n(C,L), \Q_\ell) \cong H_{et}^*(\fM_n^s(C,L), \Q_\ell) \otimes H_{et}^*(B\Gm, \Q_\ell)$$
 where  $H_{et}^*(B\Gm, \Q_\ell)$ is polynomial ring generated by a degree 2 class.
\end{proposition}


\begin{thebibliography}{99}


\bibitem{ah}
D.~Arapura, R. Hain,
{\em Torelli Group Actions on the Cohomology of Character Varieties}, in preparation 

\bibitem{bbgn}
V.~Balaji, I.~Biswas, O.~Gabber, D.~Nagaraj,
{\em Brauer obstruction for a universal vector bundle}, C.~R.\ Math.\ Acad.\ Sci.\ Paris 345 (2007), no.~5, 265–268.

\bibitem{blr}
S.~Bosch, W.~L\"utkebohmert, M.~Raynaud,
{\em N\'eron models}. Ergebnisse der Mathematik und ihrer Grenzgebiete 21, Springer-Verlag, 1990.
 

\bibitem{deligne-hodge3}
P.~Deligne,
{\em Th\'eorie de Hodge III}, Inst.\ Hautes \'Etudes Sci.\ Publ.\ Math., No.~44 (1974), 5--77.

\bibitem{dhillon}
A. Dhillon
{\em On the Cohomology of Moduli of Vector Bundles and the Tamagawa Number of $SL_n$}
Canad. J. Math. Vol. 58 (5), 2006 

\bibitem{pic}
J.-M.~Drezet, M.~Narasimhan,
{\em Groupe de Picard des vari\'t\'es de modules de fibr\'es semi-stables sur les courbes alg\'ebriques}, Invent.\ Math.\ 97 (1989), 53--94.

\bibitem{friedlander}
E. Friedlander, {\em Etale homotopy of simplicial schemes.} Ann. Math. Stud., Princeton (1982)
{\em }

\bibitem{fulton}
W. Fulton,
{\em Intersection theory}, Springer  (1998).


\bibitem{groth-chern}
A. Grothendieck,
{\em  Classes de Chern et representations lineares des groupes discretes}, Dix Expos\'es, Adv. Stud. Pure Math. (1968)


\bibitem{grothendieck}
A.~Grothendieck,
{\em Le groupe de Brauer I}.   S\'eminaire Bourbaki, Vol.~9, Exp.\ No.~290, 199--219, Soc.\ Math.\ France, Paris, 1995.


\bibitem{heinloth}
J.~Heinloth,
{\em Lectures on the moduli stack of vector bundles on a curve}. Affine flag manifolds and principal bundles, 123--153, Trends Math., Birkh\"auser/Springer, 2010. 

\bibitem{macdonald}
I. Macdonald, {\em Symmetric functions and Hall polynomials}, Oxford (1995)

\bibitem{milne}
J. Milne,
{\em \'Etale cohomology}, Princeton U. Press (1980)

\bibitem{ramanan}
S.~Ramanan,
{\em The moduli spaces of vector bundles over an algebraic curve}, Math.\ Ann.\ 200 (1973), 69--84



\bibitem{seshadri}
C.~Seshadri,
{\em Fibres vectoriels sur les courbes alg\'ebriques. } Ast\'erisque 96,  (1982).


\end{thebibliography}
\end{document}